\documentclass[a4paper]{article}
\usepackage[margin=2.5cm]{geometry} 
\usepackage{tikz}
\usepackage{subcaption} 
\usepackage{pgfplots} 
\usetikzlibrary{shapes,arrows.meta,positioning,calc}
\usepackage{float}

\pgfplotsset{compat=1.17}

\usepackage{graphicx} 

\usepackage[english]{babel}

\usepackage{amsmath}
\usepackage{amsthm}
\usepackage{graphicx}
\usepackage[colorlinks=true, allcolors=blue]{hyperref}
\usepackage{amsfonts}
\usepackage{cases}
\usepackage{bm}
\usepackage[utf8]{inputenc}
\usepackage{tikz}
\usepackage{amssymb}
\newtheorem{lemma}{Lemma}[section]
\newtheorem{proposition}{Proposition}[section]

\newtheorem{theorem}{Theorem}[section]
\newtheorem{remark}{Remark}
\newtheorem{example}{Example}
\newtheorem{corollary}{Corollary}[section]

\title{The Ollivier Ricci flow with prescribed curvature on infinite graphs}
\author{Bobo Hua, Yong Lin, Shuang Liu}
\date{}

\begin{document}
\maketitle
\begin{abstract}
In this paper,  we consider the Ricci flow with prescribed curvature on  infinite graphs, which reads as
\begin{equation*}\label{flow-equation3}
    \frac{d}{dt}\omega(t)=-(\kappa(t)-\kappa^*)\omega(t),~~ t>0,
\end{equation*}
where $\omega$ is the edge weight, $\kappa$ and $\kappa^*$ are Lin-Lu-Yau Ricci curvature and the prescribed curvature on the set of edges, respectively. First, we establish the existence and uniqueness of the solution to the Ricci flow. Furthermore, we prove the convergence of the Ricci flow for graphs  with girth at least 6 under two different conditions. Our convergence result aligns with the conclusion of Rodin and Sullivan (J Differ Geom, 26(2) 1987) that a circle packing in the plane with the hexagonal pattern is the regular hexagonal packing. 
\end{abstract}
\section{Introduction}
In 1982, Hamilton \cite{H82} introduced the Ricci flow on  a Riemannian manifold $(M,g(t))$, as follows 
\[\frac{\partial g_{ij}}{\partial t} = -2R_{ij},\]
where $R$ is the Ricci curvature. By constructing time-dependent evolution equations, the Ricci flow enables the metric of a manifold to undergo dynamic evolution and smoothing governed by its Ricci curvature. The Ricci flow played a key role in resolving the Poincaré Conjecture and the Geometrization Conjecture, see \cite{P02,2003Finite,2003Ricci}. 

Inspired by the remarkable success of curvature flows on manifolds, there has been considerable effort to generalize these flows to discrete spaces. Chow and Luo, among others, pioneered the introduction of combinatorial Ricci flow, Yamabe flow and Calabi flow  onto discrete surfaces; see, for example \cite{2002Combinatorial,ge2018combinatorial,luo2004combinatorial}. The combinatorial curvature flow can automatically resolve the existence and convergence problem for Thurston's circle packing metrics \cite{thurston1978geometry}. Subsequently, the combinatorial curvature flow was rapidly generalized to higher dimensions and more complex topological structures, and has found numerous applications; see, for example, \cite{Gu2018A,gu2018discrete} and the references therein.

In contrast to discrete surfaces, which rely on cell complexes, graphs involve only vertices and edges and do not take faces into account. In this paper, we study the Ricci flow on graphs based on Ollivier Ricci curvature, which measures differences between random walks via Wasserstein distances. The discrete Ricci flow based on Ollivier curvature $\kappa$ on graphs endowed with the edge metric $d$ was originally introduced by Ollivier \cite{2007Ricci} in 2009, as follows 
\begin{equation}\label{ollivier0}
    \frac{d}{dt}d_{xy}=-\kappa_{xy}d_{xy}.
\end{equation}
In analogy with the Ricci flow on manifolds,  the discrete Ricci flow characterizes the evolution of the edge metric governed by Ollivier Ricci curvature on the set of edges with respect to time. In 2019, Ni et al.  \cite{2019Community} effectively applied the time-discretization of the discrete Ricci flow to community detection of networks. In 2025,  Bai et al. \cite{2024Ollivier} pioneered the existence and uniqueness of the discrete Ricci flow. Subsequently, the (modified) Ricci flows  on graphs have been  greatly developed in both theory and applications, see for example  the existence and the uniqueness \cite{2024Evolution,MY24}, the convergence results \cite{BHLL25, lin2026ricci, MY25}, and the applications \cite{2025Discrete,2022Normalized,2025Community}. Besides the discrete Ricci flow based on Ollivier curvature, two other discrete Ricci flows have been developed on graphs, based respectively on  Bakry‑Émery curvature \cite{cushing2025bakry,hua2024version}, Forman curvature \cite{bai2026weighted,estrada2025forman,weber2017characterizing} and Entropy curvature \cite{erbar2020super,zhao2026efficient}. 

It is worth noting that the existing results on both combinatorial curvature flows and discrete curvature flows have been concentrated primarily on finite graphs. Recently, Ge et al. pioneered the well-posedness and convergence results for the combinatorial curvature flows on infinite triangulations \cite{ge2025combinatorial}. However, on infinite graphs, no studies concerning the discrete curvature flows have been established so far. In this paper, we focus on general infinite graphs-indepedent of any cell complex structure of dimension two, and establish the well-posed results and convergence of the solutions to the Ricci flow \eqref{ollivier0}. These results provide a discrete counterpart to the theory of Ricci flows on noncompact manifolds \cite{chen2009strong,chen2006uniqueness,giesen2011existence,shi1989deforming,topping2015uniqueness}.  Adopting the framework of  \cite{2025Ricci,lin2026ricci}, we fix the graph metric, serving only to encode the adjacency relations, while the edge weights evolve according to the Lin–Lu–Yau curvature. This choice not only elucidates the structure of the graph, but also facilitates a connection between the discrete curvature flow \eqref{ollivier0} and circle packing theory that is analogous to the combinatorial curvature flow \cite{lin2026ricci}.

Let $G=(V,E)$ be an infinite, locally finite and simple undirected graph with the set of vertices $V$ and the set of edges $E$. In this paper, we consider the following Ricci flow  with prescribed curvature based on Lin-Lu-Yau curvature (the modified Ollivier curvature) on $G$:
for any $e\in E$,
\begin{equation}\label{flow-equation3}
    \frac{d}{dt}\omega(t,e)=-(\kappa(t,e)-\kappa^*(e))\omega(t,e),~~ t>0
\end{equation}
with the positive initial value $\omega(0)$,
where $\kappa^*$ is a prescribed curvature. When $\kappa^*=0$,
it is reduced to the original Ricci flow 
\begin{equation}\label{flow-equation4}
    \frac{d}{dt}\omega(t,e)=-\kappa(t,e)\omega(t,e),~~ t>0.
\end{equation}
By choosing $\kappa^*=c$ with a constant $c\neq 0$, the Ricci flow is 
\begin{equation}\label{flow-equation_constant}
    \frac{d}{dt}\omega(t,e)=-(\kappa(t,e)-c)\omega(t,e),~~ t>0.
\end{equation}
Due to the definition of Lin-Lu-Yau curvature (see \eqref{ollivier}), we have $\kappa(a\omega)=\kappa(\omega)$ with any positive constant $a$. 
Then, $\omega(t,e)$ is the solution to  \eqref{flow-equation_constant} if and only if $e^{-tc}\omega(t,e)$ is the solution to  \eqref{flow-equation4}. 
By proving the locally Lipschitz property of $\kappa$, we first apply the method of exhaustion by finite subgraphs to establish the existence of a solution to the Ricci flow \eqref{flow-equation3}.

\begin{theorem} \label{main1} The Ricci flow \eqref{flow-equation3} on an infinite graph admits a solution $\omega_e \in C^{1,1}(0, \infty)$ for each $e \in E$ and any positive initial value $\omega(0)$. In particular, $\omega_e$ is in $C^\infty(0, \infty)$ provided the graph has girth at least 6.
\end{theorem}
For the uniqueness of of the solution to the Ricci flow \eqref{flow-equation3}, further assumptions are required. We say that \textbf{Assumption $(A)$} is satisfied if  
the following three conditions hold simultaneously:
\begin{itemize}
       \item[($A_1$)]$D := \sup_{x \in V} d_x < \infty.$
    \item[($A_2$)] $\|\kappa^*\|_\infty := \sup_{e \in E} |\kappa^*(e)| < \infty$.
    \item[($A_3$)] there exists a $\delta>0$ such that $\delta^{-1}\leq\omega(0,e)\leq\delta$ for any $e\in E$.
\end{itemize}
Under the above consistency assumption, Grönwall's Lemma immediately yields the following uniqueness of the solution.
\begin{theorem} \label{main2}
Under the assumption of $(A)$,
the solution $\omega$  to the flow \eqref{flow-equation3} on an infinite graph  for any positive initial value on $[0,\infty)$ is unique. 
\end{theorem}

Notably, Assumption $(A)$ is not necessary for all graphs. For graphs  with girth at least 6, the uniqueness of solutions to the flow \eqref{flow-equation3} can be established without requiring Assumption $(A)$. Owing to the nice properties of the Lin–Lu–Yau curvature on graphs of girth at least 6 (see \eqref{cur_tree}),  we can employ the maximum principle (Lemma \ref{maximum}) to establish uniqueness for \eqref{flow-equation3}, in close analogy with the uniqueness proof for the combinatorial curvature flow on infinite triangulations \cite{ge2025combinatorial}, as shown below.
\begin{theorem}\label{main3}
  On an infinite graph  with girth at least 6, the solution $\omega_e$ for any $e\in E$ to the Ricci flow \eqref{flow-equation3} for any positive initial value is unique.
\end{theorem}
At last, we demonstrate the convergence of the solution to the Ricci flow \eqref{flow-equation3} for graphs  with girth at least 6. Let $r=\ln \omega$. The Ricci flow \eqref{flow-equation3} can be rewritten as 
\begin{equation}\label{flow-equation_r}
    \frac{d}{dt}r(t,e)=-\kappa(t,e)+\kappa^*(e),\quad t>0, ~\forall e\in E
\end{equation}
with the initial value $r(0)(=\ln \omega(0))$. We say that the prescribed curvature $\kappa^*$ is attainable on $C(E)=\{f: E\rightarrow \mathbb{R}\}$, namely, there exists $r^*\in C(E)$ such that $\kappa(r^*)=\kappa^*$. 
We first prove the convergence of the Ricci flow \eqref{flow-equation_r} under the assumption that the initial wieght is close to the prescribed curvature weight in the sense of $\ell^\infty$ norm. Compared with the convergence results for the combinatorial curvature flow  on the hexagonal triangulation (Theorem 1.7 in \cite{ge2025combinatorial}) and hyperbolic background geometry (Theorem 1.4 in \cite{ge2025combinatorial}), we improve the smallness condition on the initial data from \(\ell^2\) to \(\ell^\infty\), and the result applies to all graphs with girth at least 6.

\begin{theorem}\label{main5}
Suppose the infinite graph has a girth of at least 6, such that the prescribed curvature $\kappa^*$ is attained at some $r^*\in \ell^\infty(E)$ and the condition $(A_1)$ is satisfied. Then, if the initial value $r(0)$ satisfies $r(0) - r^*\in \ell^2(E)$ and $\|r(0) - r^*\|_{\infty} \le 3\sqrt{3} \min\{1,(De^{2\|r^*\|_\infty})^{-2}\}$, the solution to the Ricci flow \eqref{flow-equation_r} converges to $r^*$ as $t \to \infty$. Moreover, the curvature converges to $\kappa^*$ as $t \to \infty$.
\end{theorem}

In addition, if the initial weight curvature is less than or equal to (respectively, greater than or equal to) the prescribed curvature for each edge, by the maximum principles (Lemma \ref{max_min} and Lemma \ref{max}), then the boundedness of the initial value guarantees the convergence of the Ricci flow \eqref{flow-equation_r}, as follows.

\begin{theorem}\label{main4} Consider an infinite graph  with girth at least 6. Suppose there exists $r^* \in \ell^\infty(E)$ that realizes the prescribed curvature $\kappa^*$. If the initial value $r(0)$ satisfies either:
\begin{itemize}
    \item $\sup_{e \in E} r_e(0) < \infty$ and $\kappa(r(0)) \le \kappa^*$, or
    \item $\inf_{e \in E} r_e(0) > -\infty$ and $\kappa(r(0)) \ge \kappa^*$,
\end{itemize}
then the solution to the Ricci flow \eqref{flow-equation_r} converges to some $r(\infty)\in (-\infty, \sup_{e_j\in E}r_j(0)+2\|r^*\|_\infty]$ (or $ \in [\inf_{e_j\in E}r_j(0)-2\|r^*\|_\infty,+\infty)$) realizing $\kappa^*$ and the curvature converges to $\kappa^*$ as $t \to \infty$.
\end{theorem}
\begin{remark} 
Depending on the initial data or the graph's structure, the limit $r(\infty)$ may either take the form $r^*+c\mathbf{1}$ for some constant $c \in \mathbb{R}$, or fail to lie in the affine space $r^* + \mathbb{R}\mathbf{1}$. For instance, on an infinite 3-regular tree $T_3$, given $r^*=0$ and $\kappa^*=-\frac{2}{3}$, the limit $r(\infty)$ can be either $0$ or a non-constant metric $\tilde{r}\in \ell^\infty(E)$ (see Proposition \ref{pro:T_3} for details). In contrast, on the hexagonal lattice, the limit $r(\infty)$ is necessarily equal to $r^*+c\mathbf{1}$ for some constant $c$ (see Corollary \ref{co:hexagonal}). Furthermore, if we restrict to $r^* \in \ell^p(E)$ with $p\in [1,\infty)$, the limit reduces exactly to $r(\infty)=r^*$ (see Proposition \ref{pro-curvature}).
\end{remark}
\begin{remark} 
Let the edge weight $\omega$ be the length of the edge.
   On a surface with an infinite hexagonal circle packing, from Theorem \ref{main5}, the Ricci flow \eqref{flow-equation_r} with the constant curvature weight $r^*=0$ evolves a 3-regular hexagonal lattice, with an initial value satisfying 
    $$r(0)\in \ell^2(E)\quad \mbox{and} \quad \|r(0)\|_\infty \leq \frac{\sqrt{3}}{3},$$
    into a hexagonal lattice with uniform unit edge lengths, i.e., $r(\infty)=0$. Moreover, by virtue of Theorem \ref{main4}, the same Ricci flow \eqref{flow-equation_r} with $r^*=0$ can drive a 3-regular hexagonal lattice, under another initial conditions
\[\sup_{e \in E} r_e(0) < \infty,\quad \kappa(r(0)) \le -\frac{2}{3},\]
    into a hexagonal lattice with uniform edge lengths $e^c$, where the scaling constant satisfies $c \le\sup_{e \in E} r_e(0)$. 

    This long-time asymptotic behavior perfectly agrees with the celebrated rigidity theorem of Rodin and Sullivan \cite{rodin1987convergence}, which asserts that any circle packing in the plane with a hexagonal pattern is the regular hexagonal packing. 
\end{remark}

The structure of the rest of this paper is as follows: 
In Section \ref{section_lly}, we introduce some notations and preliminaries. In Section \ref{rigity},  we study the rigidity of Lin-Lu-Yau curvature on graphs with girth at least 6.
In Section \ref{section2},  the existence and uniqueness of solution to the Ricci flow with the prescribed curvature are established. In Section \ref{section3}, we only consider graphs  with girth at least $6$, and prove the convergence of the Ricci flow with the prescribed curvature under two different conditions related to the initial value.

\section{Notations}\label{section_lly}
Let $G=(V,E)$ be an infinite graph, where $V$ is the set of vertices, and $E\subset V\times V$ is the set of edges. Denote $d_x=\#\{y\in V: y\sim x\}$ by the degree of $x\in V$.  We restrict on locally finite graph, namely, $d_x<\infty$ for any $x\in V$. Denote $N(e)$ by the union of edges incident to the endpoints $x$ and $y$ of $e\in E$. Let $E_x$ denote the set of edges incident to $x$.  
 The distance $d(x, y)$ between two vertices $x, y \in V$ is defined as the minimal number of edges connecting them. Denote $C(V)$ and $C(E)$ by the sets of functions on $V$ and $E$, separately. 
 
 The space $\ell^p(E), p\in[1,\infty)$  consists of all functions such that the $p$-th power of the absolute value is summable over $E$, as follows
$$\ell^p(E) := \left\{ u: E \to \mathbb{R} \ \bigg| \ \sum_{e \in E} |u_e|^p < \infty \right\}.$$
The $\ell^p$-norm of $u$ is defined as$$\|u\|_{\ell^p(E)} := \left( \sum_{e \in E} |u_e|^p \right)^{1/p}.$$
The space $\ell^\infty(E)$ consists of all functions  that are uniformly bounded on $E$
$$\ell^\infty(E) := \left\{ u: E \to \mathbb{R} \ \bigg| \ \sup_{e \in E} |u_e| < \infty \right\}.$$
The $\ell^\infty$-norm  of $u$ is defined as
$$\|u\|_{\ell^\infty(E)} := \sup_{e \in E} |u_e|.$$
Given $1 \le p \le q \le \infty$, the monotonicity of $\ell^p$ norms ensures that $\|f\|_{\ell^q} \le \|f\|_{\ell^p}$. Consequently, we have  $\ell^p(E) \hookrightarrow \ell^q(E)$.

Let \( x\neq y\in V \), \( \mu_x\) and \( \mu_y \) are two probability distributions defined on $V$. The wasserstein distance $W(\mu_x,\mu_y)$ is defined by 
\[W(\mu_x,\mu_y)=\inf_{A \in \Pi(\mu_x, \mu_y)} \sum_{u,v \in V} d(u,v) A(u,v),\]
where \( A : V \times V \to [0, 1] \) is the transport plan from \( \mu_x \) to \( \mu_y \), satisfying
\[
\begin{cases}
\sum_{v \in V} A(u,v) =  \mu_x (u), & u \in V, \\
\sum_{u \in V} A(u,v) =  \mu_y (v), & v \in V,
\end{cases}
\]
and the minimum is taken over all transport plans from \(  \mu_x \) to \(  \mu_x \). 

We equip the set of edges with a positive weight function $\omega\in C(E)$. Let $m(x):=\sum_{y\sim x}\omega_{xy}$.  For $\alpha > 0$,  a finitely supported probability measure on $V$ is denoted by
\begin{equation*}
    m_x^\alpha(y):=\left\{\begin{aligned}
    &\alpha,&~~ y=x,\\
    &(1-\alpha)\frac{\omega_{xy}}{m(x)},&~~y\sim x,\\
    &0,&~~\mbox{otherwise.}
\end{aligned}
  \right.
\end{equation*}
 Lin-Lu-Yau curvature is defined by,   for $x\neq y$ 
\[\kappa(x,y):=\lim_{\alpha\rightarrow 1^-}\frac{1}{1-\alpha}\left(1-\frac{W(m_x^\alpha,m_y^\alpha)}{d(x,y)}\right),\]
where $W(m_x^\alpha,m_y^\alpha)$ denotes the wasserstein distance between $m_x^\alpha$ and $m_y^\alpha$. 

   One limit-free form of Lin-Lu-Yau  curvature was proposed by Münch and Wojciechowski \cite{Florentin2017Ollivier}. For any $f\in C(V)$, the Laplacian on $C(V)$ is defined by 
\[\Delta f(x)=\frac{1}{m(x)}\sum_{z\sim x}\omega_{xz}(f(z)-f(x)),~~\forall x\in V.\]
Denote
      \[\nabla_{xy}f := \frac{f(x) - f(y)}{d(x, y)}\]
      by the gradient of $f$ with respect to $x$ and $y.$
Let $\|\nabla f\|_\infty=\sup_{x,y\in V}|\nabla_{xy}f|$. For $K \geq 0$,
\[
\operatorname{Lip}(K) := \{ f \in C(V) : \|\nabla f\|_\infty \leq K \}
\]
is the set of all $K$-Lipschitz functions on $V$ with respect to $d$.  For any $x,y\in V$, 
\begin{equation}\label{ollivier}
\kappa(x,y)=\inf_{f\in \mathcal{F} }\nabla_{xy}\Delta f,
\end{equation}
where $\mathcal{F}:=\{f\in \operatorname{Lip}(1),\nabla_{yx}f=1\}$.

Even in the case of infinite graphs, Lin-Lu-Yau  curvature for an edge $e$ on general graphs remains a local property, determined by $N(e)$ and itself.  Consequently, the curvature on infinite graphs admits a uniform bound, just as it does for finite graphs, see Lemma 2.1 in \cite{lin2026ricci}.
\begin{lemma}
The Lin-Lu-Yau Ricci curvature on infinite graphs satisfies $|\kappa_e| \leq 2$ for all  $e \in E$.
\end{lemma}\label{bounds}

\section{The  rigidity of  Lin-Lu-Yau curvature on graphs with girth at least 6}\label{rigity}
In this section, we focus on graphs with girth at least 6. There are many graphs with girth at least 6. For any graph, if you insert more than one vertex on each edge, the resulting graph will have girth at least 6.
From \eqref{ollivier}, Lin-Lu-Yau curvature on the edge $e=(x,y)$ in the graph  with girth at least 6 is
\begin{equation*}\label{cur_tree}
\kappa_{e}=2\omega_{xy}\left(\frac{1}{m(x)}+\frac{1}{m(y)}\right)-2,
\end{equation*}
see \cite{Florentin2017Ollivier} in details. Let $r=\ln \omega$. Lin-Lu-Yau curvature can be rewritten as 
\begin{equation}\label{def:curvature}
    \kappa_e(r) = 2 \left( \frac{e^{r_{xy}}}{m(x)} + 
\frac{e^{r_{xy}}}{m(y)} \right) - 2,
\end{equation}
where $m(x) = \sum_{u\sim x} e^{r_{xu}}$.
 Obviously, $\kappa$ is translation invariant with respect to $r$, namely, for any $c\in \mathbb{R},$
\[\kappa(r)=\kappa(r+c).\]
Conversely, when $V$ is finite, if $\kappa(\tilde{r}) = \kappa(r)$, then there exists a constant $c$ such that $\tilde{r} = r + c$. This can be viewed  as the rigidity of  $\kappa$, which is  established in Proposition 3.2 of \cite{lin2026ricci}. On infinite graphs with girth at least 6, the following example demonstrates that  this rigidity is not always true.   
\begin{example}
Let $P=(V,E)$ be the infinite path with $V= \{v_i\}_{i \in \mathbb{Z}}$ and  $E = \{e_i\}_{i \in \mathbb{Z}}$ with $e_i=(v_{i},v_{i+1})$. Lin-Lu-Yau curvature on $P$ satisfies
$$\frac{\kappa_i(r) + 2}{2} = \frac{1}{1 + e^{r_{i-1}-r_i}} + \frac{1}{1 + e^{r_{i+1}-r_i}}.$$
Let $r_i=0$ and $ \Tilde{r}_i =  i$ for any $e_i\in E$,  we have $\kappa(r)=\kappa(\Tilde{r})=0$.  
\end{example}

In the following proposition, the rigidity of Lin-Lu-Yau curvature is established when the solutions are restricted in  $\ell^p(E)$ with $p\in[1,\infty)$.

\begin{proposition}\label{pro-curvature}
Let $G=(V,E)$ be a graph  with girth at least $6$. For any two $r,\tilde{r}\in \ell^p(E)$ with $p\in[1,\infty)$ satisfying $\kappa(r)=\kappa(\tilde{r})$,  we have $r_i=\tilde{r}_i$
for any $e_i\in E$.
\end{proposition}
\begin{proof}
Define $u_i = \tilde{r}_i - r_i$ for all $e_i \in E$. Since $r, \tilde{r} \in \ell^p(E)$, then $u \in \ell^p(E)$.
Consequently, if $u$ is not identically zero, it must attain either a strictly positive global maximum $M>0$ or a strictly negative global minimum  $m<0$ at some edges.

Suppose that $M > 0$. Let $e_i = (x,y) \in E$ be the edge where $M$ is attained, i.e., $u_i = M$. For any  $e_j \in N_i$, we have $u_j \le u_i$. Therefore,
$$ \tilde{r}_j -  \tilde{r}_i = (r_j + u_j) - (r_i + u_i) = (r_j - r_i) + (u_j - u_i)\le r_j - r_i,$$
which yields
$$e^{\tilde{r}_j -  \tilde{r}_i} \le e^{r_j - r_i}, \quad \forall j \in  N_i.$$
Summing over  $E_x$ to get
$$\frac{1}{\sum_{e_j\in E_x}e^{\tilde{r}_j -  \tilde{r}_i}} \ge \frac{1}{\sum_{e_j\in E_x}e^{r_j-r_i}}.$$
Applying the similar argument for $E_y$, the definition of  $\kappa$ (see \eqref{def:curvature}) gives
$$\kappa_i(\tilde{r}) \ge \kappa_i(r).$$
By hypothesis, $\kappa(\tilde{r}) = \kappa(r)$. This strict equality forces all above inequalities to be exact equalities. Thus, for each $e_j \in N_i$, we must have $e^{u_j - u_i} = 1$, which strictly implies $u_j = u_i = M$. Since the graph is connected, we have
$$u_k \equiv M > 0, \quad \forall e_k \in E.$$
It follows that $\|u\|_{\ell^p}= \infty.$
This contradicts the fact that $u \in \ell^p(E)$. Thus, the assumption $M > 0$ is false, proving that $M \le 0$.

A completely symmetric argument applies to the minimum $m$, giving that $m \ge 0$. Therefore, we obtain $u\equiv 0$. This enforces $r_k = \tilde{r}_k$ for all $e_k \in E$.
\end{proof}

In $\ell^\infty(E)$, the rigidity of solutions does not necessarily hold, as follows.
\begin{proposition}\label{pro:T_3}
    The constant curvature weight restricted on $\ell^\infty(E)$ on an infinite 3-regular tree $T_3$ is not unique, thereby breaking the $\ell^\infty$-rigidity. 
\end{proposition}
\begin{proof}
    Consider $r \equiv 0$ on $T_3$. The curvature corresponding to $r \equiv 0$ on $T_3$ is given by, for any $e\in E$,
    $$\kappa_{e}(0) = 2 \left( \frac{1}{3} + \frac{1}{3} \right) - 2 = -\frac{2}{3}.$$
    Next, we construct a non-constant $\tilde{r} \in \ell^\infty(E)$ such that $\kappa_{xy}(\tilde{r}) = -2/3$ for all edges $(x,y)$. Recall the definition of Lin-Lu-Yau curvature (see \eqref{def:curvature}), it follows that 
    \begin{equation}\label{eq:mass}
        \sum_{y \sim x} \frac{m(y)}{m(x) + m(y)} = \frac{3}{2}, \quad \forall x \in V,
    \end{equation}
   where $m(x) = \sum_{y \sim x} e^{\tilde{r}_{xy}}$.
   
    We split $T_3$ into two binary branches, $T_A$ and $T_B$, by removing an edge $(A,B)$. Let $a_k$ (resp. $b_k$) denote the mass of a vertex at distance $k$ from $A$ in $T_A$ (resp. $B$ in $T_B$).  Let $a_0 = 1$ and $b_0 = 3$. Let $y_k = \frac{a_{k-1}}{a_k + a_{k-1}}$ and $z_k = \frac{b_{k-1}}{b_k + b_{k-1}}$. Solving \eqref{eq:mass} at nodes $A$ and $B$ yields $y_1 = \frac{5}{8}$ and  $z_1 = \frac{3}{8}$. For internal vertices in $T_A$ ($k \ge 1$), the equation \eqref{eq:mass} becomes
    $$\frac{a_{k-1}}{a_k + a_{k-1}} + 2 \frac{a_{k+1}}{a_k + a_{k+1}} = \frac{3}{2}.$$
 Then the above equation reduces to 
    $$y_{k+1} = \frac{1}{2}y_k + \frac{1}{4},$$
    which implies that $y_k = \frac{1}{2} +\left(\frac{1}{2}\right)^{k+2}.$
    Similarly, for $T_B$, we have $z_k = \frac{1}{2} -\left(\frac{1}{2}\right)^{k+2}$.

    In $T_A$, the ratio \[g_k = \frac{a_k}{a_{k-1}} = \frac{1-y_k}{y_k} = \frac{1 - (1/2)^{k+1}}{1 + (1/2)^{k+1}} < 1.\] 
It follows that  $a_k=\prod_{i=1}^k g_i$ is decreasing and converges to a finite $a_\infty\in (0,1)$.
 Similarly, in $T_B$, $b_k$ is increasing and converges to a finite $b_\infty > 3$. Thus, $ m(x)\in [a_\infty,b_\infty]$ for any $x\in V$. Consequently, $$\tilde{r}_{xy} = \ln \left( \frac{2}{3} \frac{m(x)m(y)}{m(x) + m(y)} \right)\in \ell^\infty(E),$$
and $\kappa_{xy}(\tilde{r}) =\kappa_{xy}(0)$ for any $(x,y)\in E$.
\end{proof}

For an edge $e_i$ with endpoints $x$ and $y$, we write $j\sim i$ (by a slight abuse of notation) to mean that edges $e_i$ and $e_j$ intersect, i.e., they share at least one of the vertices $x$ or $y$.  
In this paper, we also need to define the Laplace operator on the edge function $C(E)$, as follows
\[\Delta_\mu f(e_i)=\sum_{j\sim i}\mu_{ij}(f(e_j)-f(e_i)), \quad\forall f\in C(E),\] 
where $\mu\in C(E)$ is positive on any pair of $j\sim i$ and $\mu_{ij}=\mu_{ji}$. We say it as the $\mu$-Laplacian on $C(E)$.

Next, we establish a sufficient condition of the $\ell^\infty$-rigidity on graphs with girth at least $6$. By \eqref{def:curvature},  for any $e_i\in E$,
  \[
    \frac{\partial \kappa_{i}}{\partial r_{i}} (r)= 2\frac{e^{r_i}}{m(x)}\left(1-\frac{e^{r_i}}{m(x)}\right)+2\frac{e^{r_i}}{m(y)}\left(1-\frac{e^{r_i}}{m(y)}\right),
\]
where \( E_u \) denotes the set of edges incident to vertex \( u \). Moreover, 
 for any $j\sim i$ 
\[\frac{\partial \kappa_{i}}{\partial r_j} (r)=\frac{\partial \kappa_{j}}{\partial r_i} (r)=-2\frac{e^{r_i+r_j}}{m(x)^2}~(\mbox{or}~-2\frac{e^{r_i+r_j}}{m(y)^2}),\]
and equals to $0$ if  $e_j$ is not incident to $e_i$.  
Moreover, 
\begin{equation}\label{eq:par_kappa}
    \frac{\partial \kappa_{i}}{\partial r_{i}} (r)= -\sum_{j\sim i}\frac{\partial \kappa_{i}}{\partial r_{j}} (r).
\end{equation}
 For $r\neq \Tilde{r}$, let  $u^s:=sr + (1 - s)\Tilde{r}$ for $s\in[0,1]$. Based on \eqref{eq:par_kappa} and Newton-Leibniz formula, we have
\begin{align}\label{eq:laplace}
\kappa_i(r) - \kappa_i(\Tilde{r})&= \int_0^1 \frac{d}{ds} \kappa_i(u^s) \, ds \notag\\
&= \sum_{j \sim i} \int_0^1 \frac{\partial \kappa_i}{\partial u^s_j} (u^s) ds \cdot [r_j - \Tilde{r}_j] +\int_0^1 \frac{\partial \kappa_i}{\partial u^s_i} (u^s) ds \cdot [r_i - \Tilde{r}_i] \notag\\
&= \sum_{j \sim i} \int_0^1 \frac{\partial \kappa_i}{\partial u^s_j} (u^s) ds \cdot [(r_j- \Tilde{r}_j) - (r_i - \Tilde{r}_i)] \notag\\
&= \Delta_{\mu}(r_i - \Tilde{r}_i),
\end{align}
where for any $j\sim i$,
\begin{equation}\label{mu}
\mu_{ij} =\mu_{ij} (r,\Tilde{r})= -\int_{0}^{1} \frac{\partial \kappa_i}{\partial u^s_j}(u^s)ds\quad (> 0).
\end{equation}

Consider the line graph $L(G) = (V_L, E_L)$ of $G=(V,E)$, where $V_L=E$ and two vertices in $L(G)$ are adjacent if the corresponding edges in $G$ share a vertex. The edge weight of $L(G)$ is endowed with  $\mu$ defined as \eqref{mu}. A function $f$ on $V_L$ is called a discrete weighted harmonic (superharmonic resp.) function on the graph $L(G)$ if and only if
$$
L_\mu f(v_i):=\sum_{v_j\sim v_i}\mu_{ij}(f(v_j)-f(v_i))= 0\ (\leq 0\ \text{resp.}),\quad \forall v_i\in V_L.
$$
 We say that $L(G)$ satisfies the Liouville property for bounded harmonic functions if for every $f \in \ell^\infty(V_L)$ with $L_\mu f \equiv 0$ on $V_L$, there exists $c \in \mathbb{R}$ such that $f \equiv c$ on $V_L$.  Equivalently, $L(G)$ satisfies the Liouville property if and only if every bounded function $f$ on $E$ satisfying $\Delta_\mu f \equiv 0$ is constant. 
\begin{theorem}\label{pro:harmonic}
Let $G=(V,E,\omega)$ be a graph with girth at least $6$. For any $r,\Tilde{r}\in \ell^\infty(E)$ such that $\kappa(\Tilde{r}) = \kappa(r)$,  if the line graph $L(G) = (V_L, E_L)$ with the edge wight $\mu$ defined as \eqref{mu} satisfying the discrete Liouville property for bounded harmonic functions, then $ \Tilde{r}= r + c$ for some constant $c$.
\end{theorem}
\begin{proof}
For any $r,\Tilde{r}\in \ell^\infty(E)$, from \eqref{eq:laplace}, the condition $\kappa_i(r)=\kappa_i(\Tilde{r})$ is equivalent to 
$\Delta_{\mu} (r_i- \Tilde{r}_i)=0$ with $\mu$ defined as \eqref{mu}. 
By choosing the harmonic function $f=r- \Tilde{r}$, then $ \Tilde{r}= r + c$ for some constant $c$. 
\end{proof}

 The next Corollary is a direct conclusion of Proposition \ref{pro:harmonic}. 
\begin{corollary}\label{co:hexagonal}
    The constant curvature weight restricted on $\ell^\infty(E)$ on the hexagonal lattice $H$ is  unique up to scaling. 
\end{corollary}
\begin{proof}
  For any $r,\Tilde{r}\in \ell^\infty(E)$, notice that $\mu_{ij}$ defined as \eqref{mu} has uniform bounds; namely, $\mu_{ij}\in(0,2]$. Consider the line graph $L(H)$ of the hexagonal lattice $H$, which is 4-regular Kagome lattice. Fix $x_0\in V_L$, denote the ball $B_r=\{x\in V_L: d(x,x_0)\leq r\}$ with $r\ge 1$, we have 
  \[\mu(B_r)=\sum_{x\in B_r}\sum_{v_j\sim v_i}\mu_{ij}\le 8|B_r|\leq Cr^2.\]
By Theorem 6.13 in \cite{grigor2018introduction}, the simple random walk is recurrent on $L(H)$. Due to Theorem 1.16 in \cite{woess2000random}, this yields that any positive (super)harmonic function on $L(H)$ corresponding to $\mu$ is constant. Combing this with  Proposition \ref{pro:harmonic}, we obtain there exists a constant $c$ such that $ \Tilde{r}= r + c$ if $\kappa_i(r)=\kappa_i(\Tilde{r})$ for any $e_i\in E$.
\end{proof}

\section{The existence and uniqueness of solution to the Ricci flow \eqref{flow-equation3}}\label{section2}
Analogous to the finite case, Lin–Lu–Yau curvature on infinite graphs, which is locally defined (see \eqref{ollivier}), satisfies the following Lipschitz continuity.  Set $N_i=\{e_j\in E: j\sim i\}$ and $\overline{N}_i=N_i\cup\{e_i\}$.   For simplicity, we denote $f_i=f(e_i)$ for any $f\in C(E)$.

\begin{proposition}\label{lip_kappa}
For any $e_i\in E$, $\kappa_i( \omega)$ is Lipschitz with respect to $ \omega$ on $\overline{N}_i$, i.e., for any two $ \omega\neq \Tilde{ \omega}$,
 \begin{equation*}\label{eq:kappa_lip}
    |\kappa_i( \omega)-\kappa_i(\Tilde{ \omega})|\leq L\| \omega-\Tilde{\omega}\|_\infty^i,
\end{equation*}
where $\| \omega-\Tilde{\omega}\|_\infty^i=\sup_{e_j\in \overline{N}_i}|\omega_j-\Tilde{\omega_j}|$. Let $e_i=(x,y)$ and $D_i=\max(d_x,d_y)$.
The Lipschitz constant $L$ depends on $\omega_i, \sup_{e_j\in\overline{N}_i}\omega_j$ and $D_i$.
\end{proposition}
\begin{proof}
    Let $ \omega\neq \Tilde{ \omega}$.
Let $f$ and $\Tilde{f}$ be the optimal functions that attain the infimum in \eqref{ollivier} for $ \omega$ and $\Tilde{ \omega}$, respectively. Let $\Tilde{\Delta}$ and $\Tilde{\Delta}$ be the Laplace operators associated with $\omega$ and $\Tilde{ \omega}$. Thus, for  $e_i=(x,y)$, 
  \begin{align}\label{eq:kappa_gradient}
      |\kappa_i( \omega)-\kappa_i(\Tilde{ \omega})|
      &=|\nabla_{xy}\Delta f-\nabla_{xy}\Tilde{\Delta} \Tilde{f}|\notag \\
      &\leq \max(|\nabla_{xy}\Delta f-\nabla_{xy}\Tilde{\Delta} f|,|\nabla_{xy}\Delta \Tilde{f}-\nabla_{xy}\Tilde{\Delta} \Tilde{f}|).
  \end{align}
Since 
\[m(x)=\sum_{e_j\in E_x}\omega_{xu}\in [\omega_i,d_x\sup_{e_j\in E_x}\omega_j],\]
 and
\[|\Tilde{m}(x)-m(x)|\leq |\sum_{e_j\in E_x}|\Tilde{\omega}_j-\omega_j|\leq d_x\| \omega-\Tilde{\omega}\|_\infty^i.\]
we obtain, for any $e_j\in E_x,$
\begin{align*}
    |\Tilde{m}(x)\omega_j-m(x)\Tilde{\omega}_j|&\leq m(x)|\omega_j-\Tilde{\omega}_j|+\omega_j|\Tilde{m}(x)-m(x)|\\
    &\leq 2d_x\sup\limits_{e_j\in E_x}\omega_j\cdot \| \omega-\Tilde{\omega}\|_\infty^i,
\end{align*}
Therefore, due to $f\in\operatorname{Lip}(1),$
\begin{align*}|\nabla_{xy}\Delta f-\nabla_{xy}\Tilde{\Delta} f|
&=|\Delta f(x)-\Delta f(y)-\Tilde{\Delta}f(x)+\Tilde{\Delta} f(y)|\\
&\leq\sum_{u\sim x}\left|\frac{\omega_{xu}}{m(x)}-\frac{\Tilde{\omega}_{xu}}{\Tilde{m}(x)}\right||f(u)-f(x)|+\sum_{v\sim y}\left|\frac{\omega_{yv}}{m(y)}-\frac{\Tilde{\omega}_{yv}}{\Tilde{m}(y)}\right||f(v)-f(y)|\\
&\leq\sum_{e_j\in E_x}\frac{|\Tilde{m}(x)\omega_j-m(x)\Tilde{\omega}_j|}{m(x)\Tilde{m}(x)}+\sum_{e_j\in E_y}\frac{|\Tilde{m}(y)\omega_j-m(y)\Tilde{\omega}_j|}{m(y)\Tilde{m}(y)}\\
&\leq L\| \omega-\Tilde{\omega}\|_\infty^i,
\end{align*}
where $L=4D_i^2\omega_i^{-2}\sup_{e_j\in \overline{N}_i}\omega_j$.
Similarly, $|\nabla_{xy}\Delta \Tilde{f}-\nabla_{xy}\Tilde{\Delta}\Tilde{f}|$ admits the same bound. Substitude the bounds into \eqref{eq:kappa_gradient} to
 complete the proof of the lemma.
\end{proof}
We construct an exhaustion of the graph by connected finite subgraphs.  Fix a vertex $x_0 \in V$ and an integer $l >0$. By the local finiteness of $G$, it follows that $G^l = (B^l, E^l, \omega^l)$ is a finite and connected subgraph of $G$, where
$$B^l = \{x \in V : d(x, x_0) < l\},\quad E^l = \{(x,y) \in E : x, y \in B^l\}.$$
The boundaries of $B^l$ and $E^l$ are defined by$$\partial B^l = \{x \in V : d(x, x_0) =l\},\quad \partial E^l = \{(x,y) \in E : x\in B^l , y\in \partial B^l \}.$$
Hence, the sequence $\{E^l\}_{l> 0}$ satisfies
$$E^l\subset E^{l+1}, \quad \bigcup_{l > 0} E^l = E,$$and then $\{G^l\}_{l > 0}$ is called an exhaustion sequence of the graph $G$.

The Ricci flow with prescribed curvature on $E^l\cup \partial E^l$ is defined as follows
\begin{equation}
\begin{cases} 
\displaystyle \frac{\mathrm{d} \omega_i^{l}(t)}{\mathrm{d} t} = -\left(\kappa_i^l-\kappa_i^*\right)\omega^l_i(t), \quad &\forall e_i \in E^l, \ \forall t > 0, \\ 
\omega_i^{l}(t) = \omega_i(0), \quad &\forall (i,t) \in (E^l \times \{0\}) \cup (E^{l+1}\setminus E^l \times (0, \infty)),
\end{cases} \label{flow_exhaustion}
\end{equation}
where $\kappa_i^l=\kappa_i(\omega^l(t))$ is the Lin-Lu-Yau Ricci curvature for $e_i\in E_l$.
Since $E_l$ is finite, the local existence and uniqueness of the flow \eqref{flow_exhaustion} follow from the Picard-Lindelöf theorem on a finite time interval.
And, the long time existence of the flow \eqref{flow_exhaustion} by virtue of the boundedness of $\kappa$ is similar to that of finite graph, see \cite{lin2026ricci}. 

\begin{proof}[Proof of Theorem \ref{main1}]
First, we {\it claim} that for any $e_i \in E$ and any sufficiently large $l$ such that $e_i \in E^{l}$, the function $\frac{d\omega_i^l}{dt}(t)$ is Lipschitz continuous with respect to $t$ on $[0, T]$ for any $T > 0$. The Lipschitz constant is uniform with respect to $l$. 
Integrating the Ricci flow
$$\frac{d}{dt} \ln \omega_i^l(t) = -\kappa_i^l+\kappa_i^*$$
from $0$ to $T$, by the bounds of $\kappa_i^l$ in Lemma \ref{bounds},
 we obtain 
\begin{equation}\label{eq:solu_bound}
    \omega_i(0) e^{-(2+|\kappa_i^*|)T} \leq \omega_i^l(t) \leq \omega_i(0) e^{(2+|\kappa_i^*|)T}
\end{equation}
for any $e_i\in E^{l}$. It follows that on $[0,T]$,
\[m^l(x)=\sum_{e_j\in E_x}\omega^l_{xu}(t)\in \left[\omega_i(0) e^{-(2+|\kappa_i^*|)T} ,d_x\sup_{e_j\in E_x}\omega_j(0) e^{(2+|\kappa_j^*|)T}\right].\]
Notice that 
$\|\omega^l-\Tilde{\omega}^l\|_\infty^i=\|\omega^l-\Tilde{\omega}^l\|_\infty$. From Proposition \ref{lip_kappa}, 
we have
 \begin{equation}\label{eq:kappa_lip}
    |\kappa_i^l( \omega^l)-\kappa_i(\Tilde{ \omega}^l)|\leq L_1\| \omega^l-\Tilde{\omega}^l\|_\infty,
\end{equation}
where $L_1$ is free of $l$, and depends on $\omega_i(0)$, $\sup_{e_j\in \overline{N}_i}\omega_j$,  $\sup_{e_j\in \overline{N}_i}|\kappa_j^*|$, $D_i$ and $T$. By Mean Value Theorem, for any $t_1, t_2 \in [0, T]$, there exists some $\xi$ between $t_1$ and $t_2$ such that
$$|\omega^{l}_i(t_1) - \omega^{l}_i(t_2)| = \left|\frac{d\omega_i^l}{dt}(\xi)\right| \cdot |t_1 - t_2|.$$
Since 
\begin{equation}\label{1_bounds}
    \left|\frac{d\omega_i^l}{dt}(t)\right| = |\kappa^{l}_i-\kappa^*_i| \cdot \omega^l_i(t)\leq (2+|\kappa^*_i|)\omega_i(0) e^{(2+|\kappa_i^*|)T},
\end{equation}
we  obtain
\begin{equation*}\label{eq:sol_lip}
    \|{\omega}^l(t_1) - {\omega}^l(t_2)\|_\infty\leq (2+|\kappa^*_i|)\omega_i(0) e^{(2+|\kappa_i^*|)T} |t_1 - t_2|.
\end{equation*} 
It follows that, by \eqref{eq:kappa_lip}
\begin{equation*}\label{eq_kappa}
    |\kappa_i({\omega}^l(t_1)) - \kappa_i({\omega}^l(t_2))|\leq L_2|t_1 - t_2|,
\end{equation*}
and then
\begin{align}\label{lip_omega}
\left|\frac{d\omega_i^l}{dt}(t_1)-\frac{d\omega_i^l}{dt}(t_2) \right| 
&= |\kappa_i({\omega}^l(t_1))\omega_i^l(t_1) - \kappa_i({\omega}^l(t_2))\omega_i^l(t_2)-\kappa_i^*(\omega_i^l(t_1)-\omega_i^l(t_2))| \notag\\
&\le \left|\kappa_i({\omega}^l(t_1)) - \kappa_i({\omega}^l(t_2))\right|\cdot \omega_i^l(t_1) + \left(|\kappa_i({\omega}^l(t_2))|+|\kappa^*_i|\right) \, \|{\omega}^l(t_1) - {\omega}^l(t_2)\|_\infty \notag\\
&\leq L_3|t_1-t_2|,
\end{align}
where $L_2$ and $L_3$ depend on $\omega_i(0)$,  $\sup_{e_j\in \overline{N}_i}\omega_j$,  $\sup_{e_j\in \overline{N}_i}|\kappa_j^*|$, $D_i$ and $T$. This completes the proof of the claim.

On $[0,T]$,
combining \eqref{eq:solu_bound},\eqref{1_bounds} and \eqref{lip_omega}, and by applying  the Arzel\`{a}-Ascoli theorem and the standard diagonal argument, there exists a 
subsequence $\{\omega_i^{l_k}(t)\}_{k=1}^\infty$ of $\{\omega_i^{l}(t)\}_{l=1}^\infty$ on $[0,T]$ for any $e_i\in E$ such that 
$\omega_i^{l_k}(t)$ converges in $C^1[0, T]$ to some $\omega_i(t)$ as $k\to \infty$ . 
By virtue of \eqref{eq:kappa_lip},  $\kappa_i(\omega^{l_k}(t))$ converges in $C^0[0, T]$ to $\kappa_i(\omega(t))$ as $k\rightarrow\infty$
for any $e_i\in E$.
Therefore, $\omega_i(t)$ for any $e_i\in E$ satisfies the flow \eqref{flow-equation3} on $[0,T]$. 

From the uniform boundedness of $\kappa$ and $\kappa^*$, the solution $\omega(t)$ can be globally extended to $(0, \infty)$. Moreover, from the definition of $\kappa$ (see \eqref{ollivier}), the solution $\omega(t)$ is $C^{1,1}$ with respect to $t$.
In particular, on graphs  with girth at least 6, $\kappa$ (see \eqref{cur_tree}) is smooth with respect to $\omega$. This implies that the solution $\omega_i(t)$ is smooth with respect to $t$ on graphs  with girth at least 6.
\end{proof}
\begin{lemma}[Grönwall's Lemma]\label{lem:Grönwall}
    Assume that $f(t)$ is a non-negative continuous function satisfying $$f(t) \le C + \int_a^t g(s)u(s) \, ds$$where $C$ is a constant and $g(s)$ is a non-negative integrable function. Then, $$f(t) \le C \exp \left( \int_a^t g(s) \, ds \right).$$
   In particular, $C=0$ immediately yields $f(t) \equiv 0$ on $[a, T]$.
\end{lemma}

\begin{proof}[Proof of Theorem \ref{main2}] For any $T>0$, let $\omega$ and $\Tilde{\omega}$ be two different positive function on $C(E)$. 
 Similar to the proof of  Proposition \ref{lip_kappa}, we have for any $e_i\in E$,
\begin{equation*}
    |\kappa_i( \omega)-\kappa_i(\Tilde{ \omega})|\leq L_1\left(\omega_i, \sup_{e_j\in\overline{N}_i}\omega_j, D_i\right)\| \omega-\Tilde{\omega}\|_\infty^i.
\end{equation*}
Let $ f_i(\omega):=-(\kappa_i( \omega)-\kappa^*_i)\omega_i$.
It follows that
\[\begin{aligned}
|f_i(\omega)-f_i(\Tilde{ \omega})| &\le \left|\kappa_i({\omega}) - \kappa_i({\Tilde{\omega}})\right|\cdot \omega_i + \left(|\kappa_i({\Tilde{\omega}})|+|\kappa^*_i|\right) \, |{\omega}_i - {\Tilde{\omega}}_i|\\
&\leq L_4\left(\omega_i, \sup_{e_j\in\overline{N}_i}\omega_j, D_i,|\kappa_i^*|\right)\| \omega-\Tilde{\omega}\|_\infty^i.
\end{aligned}\]
Assume there exist two solutions $\omega^{(1)}(t)$ and $\omega^{(2)}(t)$ satisfying the same initial condition $\omega^{(1)}(0) = \omega^{(2)}(0)$. On $[0,T]$, similar to \eqref{eq:solu_bound}, under the assumption of $(A_2)$ and $(A_3)$, 
\begin{equation}\label{eq:omega_bounds_uniform}
    \omega_i^{(k)}(t)\in[\delta^{-1} e^{-(2+\|\kappa^*\|_\infty)T},\delta e^{(2+\|\kappa^*\|_\infty)T}]
\end{equation}
for any $e_i\in E$ and $k=1,2$. 
Therefore, under the assumption of $(A)$, 
\begin{equation}\label{eq:f_lip}
\left|f_i(\omega^{(1)}) - f_i(\omega^{(2)}) \right|\leq L_5\| \omega^{(1)}-\omega^{(2)}\|_\infty^i
\end{equation}
on $[0,T]$, where $L_5$ is uniform and depends on $\delta,D,\|\kappa^*\|_\infty$ and $T.$

 We define $\phi_i(t) = \omega^{(1)}_i(t) - \omega^{(2)}_i(t)$ for any $e_i\in E$. From \eqref{eq:omega_bounds_uniform}, for any $t \in [0, T]$ and any $e_i \in E$, $\phi_i(t)$ is uniformly bounded, i.e., $ \|\phi(t)\|_{\ell^\infty} <\infty$.
 From the Ricci flow \eqref{flow-equation3}, we have for  each $e_i\in E$,
 $$ \phi_i(t) = \int_0^t \left[ f_i(\omega^{(1)}(s)) - f_i(\omega^{(2)}(s)) \right] ds $$
 on $[0,T]$.
 Taking the absolute value on both sides and substituting \eqref{eq:f_lip} to get
 $$ |\phi_i(t)| \le \int_0^t \left|f_i(\omega^{(1)}(s)) - f_i(\omega^{(2)}(s)) \right| ds \leq L_5\int_0^t\|\phi(s)\|_{\ell^\infty} ds. $$
It follows that 
 $$ \|\phi(t)\|_{\ell^\infty} \leq L_5\int_0^t\|\phi(s)\|_{\ell^\infty} ds. $$
According to Grönwall's Lemma (see Lemma \ref{lem:Grönwall}), we  conclude
$$ \|\phi(t)\|_{\ell^\infty}= 0, \quad \forall t \in [0, T]. $$
This implies $\omega^{(1)}_e(t) = \omega^{(2)}_e(t)$ for all edges $e$ and all $t \in [0, T]$. 
Moreover, similar to the case of finite graphs,  the solution $\omega$ on $[0,T]$ can be extended to $[0,\infty)$ by the uniformly boundedness of $\kappa$ and $\kappa^*$.
\end{proof}

The following Lemma  is the key to prove the uniqueness of the Ricci flow on the graph with girth at least 6, which is also referred to as the maximum principle. The corresponding result for functions on vertices has already been proved in \cite{ge2025combinatorial}. In contrast, we focus on functions defined on edges on $G$. However, by considering its line graph $L(G)$ and the vertex Laplacian on $L(G)$, taking the edge weights of $G$ as the vertex weights of $L(G)$, and using similar proofs from the literature, the following conclusion can be obtained. The proof is omitted here;  see Corollary 3.9 of \cite{ge2025combinatorial}. 
\begin{lemma}\label{maximum}
If $g(t)$ is a bounded solution to the equation
\begin{equation*}
\frac{dg(t)}{dt} = \Delta_{\mu(t)}g + hg
\end{equation*}
with $g(0) \equiv 0$ on $E$, where $\mu(t)$ satisfies 
\begin{equation}\label{mu_bound}
    \sum_{j \sim i} \mu_{ij}(t) \le C, \quad \forall (i, t) \in E \times [0, T]
\end{equation}
for some constant $C$,  and the function $h \le B$ for some constant $B$, then for any $t \ge 0$, $g(t) \equiv 0$.
\end{lemma}
\begin{proof}[Proof of Theorem \ref{main3}] 
 Let $r$ and $\Tilde{r}$ be two solutions of the flow \eqref{flow-equation_r} with the same initial value $r(0)$ on $G$. Let  $u^s(t):=sr(t) + (1 - s)\Tilde{r}(t)$ for $s\in[0,1]$. Then for any $e_i\in E$,
\begin{equation*}
\frac{d(r_i(t) - \Tilde{r_i}(t))}{dt} = -(\kappa_i(r(t)) - \kappa_i(\Tilde{r}(t))).
\end{equation*}
For a fixed $t$, similar to \eqref{eq:laplace},
\[
\kappa_i(r(t)) - \kappa_i(\Tilde{r}(t))= -\Delta_{\mu(t)}(r_i(t) - \Tilde{r_i}(t))
\]
where for any $j\sim i$,
\begin{equation*}\label{mu_t}
\mu_{ij}(t) = -\int_{0}^{1} \frac{\partial \kappa_i}{\partial u^s_j}(u^s(t))ds.
\end{equation*}
It yields that for any $t>0$ and $e_i\in E$,
\begin{equation*}
\frac{d(r_i(t) - \Tilde{r_i}(t))}{dt} = \Delta_{\mu(t)}(r_i(t) - \Tilde{r_i}(t)).
\end{equation*}
 
Since
\[-\sum_{e_j\in E_x,j\neq i}\frac{\partial \kappa_i}{\partial u^s_j}(u^s)=2 e^{u^s_i}\frac{\sum_{e_j\in E_x,j\neq i}e^{u^s_j}}{\left(\sum_{e_k\in E_x} e^{u^s_k}\right)^2}\leq 2,\]
we have
\[\sum_{j \sim i} \mu_{ij}(t) =-\int_{0}^{1} \sum_{j \sim i} \frac{\partial \kappa_i}{\partial u^s_j}(u^s(t))ds\leq \int_{0}^{1} \left(\sum_{e_j\in E_x,j\neq i}\frac{\partial \kappa_i}{\partial u^s_j}(u^s(t))+\sum_{e_j\in E_y,j\neq i}\frac{\partial \kappa_i}{\partial u^s_j}(u^s(t))\right)ds\leq 4.\]
Let $g(t)=r(t) - \Tilde{r}(t)$ and $h=0$, by applying  Lemma \ref{maximum} to $g(t)$, we can conclude that $r_i(t) = \Tilde{r}_i(t)$ for each $e_i\in E$ on $[0,T]$. This completes the proof.
\end{proof}


\section{Convergence of the flow \eqref{flow-equation3} on graphs with girth at least 6}\label{section3}
By the definition of Lin-Lu-Yau curvature, for any $e_i\in E$, $\kappa_i$ is determined entirely by $\overline{N}_i$ and $\kappa_i(r):\mathbb{R}^{|\overline{N}_i|}\rightarrow \mathbb{R}$. 
Denote $J_{ij}(r)=\frac{\partial \kappa_{i}}{\partial r_j}(r)$ for any $e_i,e_j\in E$. Let $e_i=(x,y)$. Notice that

\[J_{ij}(r)=
\begin{cases} 
2\frac{e^{r_i}}{m(x)}\left(1-\frac{e^{r_i}}{m(x)}\right)+2\frac{e^{r_i}}{m(y)}\left(1-\frac{e^{r_i}}{m(y)}\right), \quad &e_j=e_i,\\
-2\frac{e^{r_i+r_j}}{m(x)^2},\quad &e_j\in E_x\setminus\{e_i\},\\ 
-2\frac{e^{r_i+r_j}}{m(y)^2}, \quad &e_j\in E_y\setminus\{e_i\},\\ 
0, \quad &\mbox{otherwise,}
\end{cases} 
\]
where $m(x) = \sum_{e_k\in E_x} e^{r_k}$.

\subsection{Proof of Theorem \ref{main5}}

Denote  $u\in C(E)$ and  $r=u+r^*$ on $E$. Expanding $\kappa_i(r)$ in a first-order Taylor series around $r^*$ to get
\begin{equation}\label{eq:tylor}
    \kappa_i(r) = \kappa_i(r^*) + \sum_{e_k \in E}  J_{ik}(r^*) u_k + F_i(u).
\end{equation}
Notice that
\begin{equation*}\label{eq:trans}
    \sum_{e_k\in E} J_{ik} u_k = J_{ii} u_i + \sum_{k \neq i} J_{ik} u_k = -\sum_{e_k\in N_i} J_{ik} u_i + \sum_{e_k\in N_i} J_{ik} u_k = \sum_{e_k\in N_i} J_{ik} (u_k-u_i).
\end{equation*}
Then 
$$ F_i(u):=\kappa_i(u+r^*)-\kappa_i(r^*) - \sum_{e_k \in N_i}  J_{ik}(r^*) (u_k-u_i) .$$

\begin{lemma}\label{lemma:3.5}
For any $e_i\in E$, 
\begin{equation}\label{eq:F_upper}
    |F_i(u)| \le \frac{\sqrt{3}}{18} \sum_{e_k\in N_i} (u_k-u_i)^2 .
\end{equation}
\end{lemma}
\begin{proof}
For a fixed $e_i\in E$, let $z_k=u_k - u_i$ for any $e_k\in \overline{N}_i$. Notice that $z_i=0$. Sort the edges in \(N_i\), obtaining \(N_i=\{e_i^1,\cdots,e_i^{|N_i|}\}\). Let $z_k=z(e_i^k)$ and $r_k=r(e_i^k)$.
Denote the vector $z=(z_1,\cdots,z_{|N_i|})$. Since $\kappa_i$ possesses translation invariance, i.e. $\kappa(r)=\kappa(r+c)$ for any constant function $c$, 
    $F_i(u)$ can be rewritten as $\tilde{F}_i:\mathbb{R}^{|N_i|}\rightarrow\mathbb{R}$ by
\begin{equation}\label{eq:F}
    \tilde{F}_i(z) = \kappa_i(r^* + z) - \kappa_i(r^*) - \sum_{e_k \in N_i} J_{ik}(r^*)  z_k.
\end{equation}
Moreover,  we have
$$\tilde{F}_i(0) = 0.$$
And,  for any $e_k \in N_i$,
\begin{equation}\label{gradient_F}
    \frac{\partial \tilde{F}_i}{\partial z_k}(0) = \left.J_{ik} (r^*+z)\right|_{z=0} - J_{ik} (r^*) = 0,
\end{equation}
which implies that $\nabla \tilde{F}(0) =0$. 
Let $\xi = \theta z$ with $0 < \theta < 1$, we have
\begin{equation*}\label{eq:bound_F}
    \tilde{F}_i(z) = \tilde{F}_i(0) + \nabla \tilde{F_i}(0) \cdot z + \frac{1}{2} z^T H_i(\xi) z=\frac{1}{2} z^T H_i(\xi) z,
\end{equation*}
where $H_i(\xi)=(H_i^{kl}(\xi))_{m\times m}$ is the Hessian matrix of $\tilde{F}$ at $\xi$. We {\it claim} that 
 \begin{equation}\label{eq:hessian}
 \left| H_i^{kl}(z) \right| \le \frac{\sqrt{3}}{9}, \quad \forall e_i^k,e_i^l\in N_i.
 \end{equation}
By virtue of  $ |z_kz_l| \le \frac{1}{2}(z_k^2 + z_l^2) $,  we have
 \begin{equation}\label{eq:F_upper_z}
    |\tilde{F}_i(z)| \le \frac{\sqrt{3}}{18} \sum_{e_k\in N_i} z_k^2,
\end{equation}
which yileds \eqref{eq:F_upper}.
 It remains to prove the claim.
From \eqref{eq:F}, 
$$H_i^{kl}(z) = \left. \frac{\partial^2 \kappa_i}{\partial r_k \partial r_l}(r)\right|_{r = r^* + z}.$$
Denote $e_i=(x,y)$. When $e_i^k\neq e_i^l\in E_x$, 
$$\frac{\partial^2 \kappa_i}{\partial r_k \partial r_l}(r) =4\frac{e^{r_i+r_k+r_l}}{m(x)^3}.$$
When $e_i^k=e_i^l\in E_x$,
$$\frac{\partial^2 \kappa_i}{\partial r_k ^2}(r) =-2\frac{e^{r_i+r_k}}{m(x)^2}\left(1-\frac{2e^{r_k}}{m(x)}\right).$$
 If $e_i^k,e_i^l$ is incident to $y$, we have similar results by repalcing $m(x)$ to $m(y)$. Otherwise, $\frac{\partial^2 \kappa_i}{\partial r_k \partial r_l}(r)=0$. Notice that the facts of $\frac{e^{r_j}}{m(x)}\in (0,1)$, $\frac{e^{r_i}+e^{r_k}+e^{r_l}}{m(x)}\in(0,1]$ for $e^{r_i},e^{r_k},e^{r_l}\in E_x$, 
$$x y z \le \left( \frac{x+y+z}{3} \right)^3 \le \left(\frac{1}{3}\right)^3 = \frac{1}{27},\quad \forall x+y+z\in(0,1].$$
 and
\begin{equation*}\label{eq:estimate}
    |xy(1-2x)|\leq |x(1-x)(1-2x)|\leq \frac{\sqrt{3}}{18}, \quad \forall x\in(0,1), x+y\in(0,1].
\end{equation*}
They lead directly to the claim.
\end{proof}

From Lemma \ref{lemma:3.5}, we have the following conclusion.
\begin{proposition}\label{lemma:3.52}
Under the assumption of $(A_1)$, there exist $C_1,C_2>0$ depending on $D$ such that for any $u\in \ell^2(E)$,
\begin{equation}\label{eq:F_bound}
\|F(u)\|_{\ell^2} \le C_1 \|u\|_{\ell^2}^2,
\end{equation}
Moreover, for any $u,v\in \ell^2(E)$, 
\begin{equation}\label{eq:F_bound_gradient}
\|F(u) - F(v)\|_{\ell^2} \le C_2\max(\|u\|_{\ell^2},\|v\|_{\ell^2}) \|u - v\|_{\ell^2}.
\end{equation}
\end{proposition}
\begin{proof}

From \eqref{eq:F_upper} and $|N_i|\leq 2(D-1)$, for any $ e_i\in E$, 
 $$|F_i(u)| \le \frac{\sqrt{3}}{9} \sum_{e_k \in N_i} (u_i^2 + u_k^2) \le \frac{2\sqrt{3}(D-1)}{9}  u_i^2 + \frac{\sqrt{3}}{9}\sum_{e_k \in N_i} u_k^2\leq C\sum_{e_k \in \overline{N}_i} u_k^2,$$
 where $C=\frac{2\sqrt{3}(D-1)}{9}$.
Therefore, applying Cauchy-Schwarz inequality to yield
$$\|F(u)\|_{\ell^2}^2 = \sum_{e_i \in E} |F_i(u)|^2 \le C^2(2D-1) \sum_{e_i \in E} \sum_{e_k \in \overline{N}_i} u_k^4 .$$
Rerranging the sum to yield
$$\sum_{e_i \in E} \sum_{e_k \in \overline{N}_i} u_k^4 \leq 2(D-1)\sum_{e_k\in E}u_k^4,$$
combining with the fact that $\|u\|_{\ell^2} \ge \|u\|_{\ell^4}$ to get \eqref{eq:F_bound}.

As for the second result, for any $u,v\in \ell^2(E)$, we define the corresponding vectors $z^{(u)}, z^{(v)} \in \mathbb{R}^m$ for each fixed $e_i\in E$, as follows
$$z_k^{(u)} = u_k - u_i, \quad z_k^{(v)} = v_k - v_i, \quad \forall e_k \in N_i$$
We construct the line $z(s) = z^{(v)} + s(z^{(u)} - z^{(v)})$ with $s \in [0, 1]$ connecting $z^{(v)}$ to $z^{(u)}$. Notice that $z_i^{(u)}=z_i^{(v)}=z_i(s)=0.$
For any $e_i\in E$,
\begin{equation}\label{eq:inter}
    \tilde{F}_i(z^{(u)}) - \tilde{F}_i(z^{(v)}) = \int_0^1 \sum_{e_k \in N_i} \frac{\partial \tilde{F}_i}{\partial z_k}(z(s)) \cdot \left( z_k^{(u)} - z_k^{(v)} \right) \, ds
\end{equation}
From \eqref{gradient_F}, we have 
 $$\frac{\partial \tilde{F}_i}{\partial z_k}(z(s)) = \int_0^1 \sum_{e_l \in N_i} H_{kl}(\tau z(s)) \cdot z_l(s) \, d\tau.$$
Combining with \eqref{eq:hessian} to get
 $$\left| \frac{\partial \tilde{F}_i}{\partial z_k}(z(s)) \right| \le \frac{\sqrt{3}}{9}\sum_{e_l \in N_i} |z_l(s)|.$$
Substituting it into \eqref{eq:inter} and applying Cauchy-Schwarz inequality, we obtain
\[\begin{split}
    |\tilde{F}_i(z^{(u)}) - \tilde{F}_i(z^{(v)})|^2 &\le \frac{3}{81} \int_0^1 \left( \sum_{e_l \in N_i} |z_l(s)| \right)^2 \left( \sum_{e_k \in N_i} \left| z_k^{(u)} - z_k^{(v)} \right| \right)^2 ds\\
    &\le \frac{12}{81}(D-1)^2 \int_0^1 \left( \sum_{e_l \in N_i} z_l(s)^2 \cdot \sum_{e_k \in N_i} (z_k^{(u)} - z_k^{(v)})^2 \right) ds.
\end{split}\]
Let $h(s)=v+s(u-v)$, we have
$$\sum_{e_l \in N_i} z_l(s)^2 = \sum_{e_l \in N_i} (h_l(s) - h_i(s))^2 \le 2 \sum_{e_l \in N_i} (h_l(s)^2 + h_i(s)^2) \le 4(D-1) \|h(s)\|_{\ell^2}^2,$$
and 
$$\int_0^1 \|h(s)\|_{\ell^2}^2 \, ds \le \max(\|u\|_{\ell^2}^2,\|v\|_{\ell^2}^2).$$
Similarly, let $\delta=u-v$,
$$\sum_{e_k \in N_i} (z_k^{(u)} - z_k^{(v)})^2 = \sum_{e_k \in N_i} (\delta_k - \delta_i)^2 \le 2(D-1) \|\delta\|_{\ell^2}^2 .$$
It follows that
$$\|F(u) - F(v)\|_{\ell^2}  \le C_2 \max(\|u\|_{\ell^2},\|v\|_{\ell^2})\|u - v\|_{\ell^2},$$
where $C_2>0$ depends on $D$, $\|u\|_{\ell^2}$ and $\|v\|_{\ell^2}.$
This completes the proof.
\end{proof}

Denote the energy by 
$$\mathcal{E}(u) = \sum_{e_i,e_j \in E} J_{ij}(r^*) u_i u_j$$
and the $\ell^2$-norm of the gradient of $u$ by
$$\|\nabla u\|_{\ell^2}^2 = \sum_{e_i\in E}\sum_{j\sim i} (u_i - u_j)^2.$$
\begin{proposition}\label{lemma:3.4}
Suppose that $r^*\in\ell^\infty(E)$ and $(A_1)$. There exists $C_3>0$ depending on $D$ and $\|r^*\|_\infty$ such that
    $$\mathcal{E}(u) \ge  C_3 \|\nabla u\|_{\ell^2}^2.$$
\end{proposition}
\begin{proof}
Arrange the sum of $\mathcal{E}(u)$  to yield
$$\mathcal{E}(u) = \sum_{x \in V} \left(  \sum_{e_i, e_j \in E_x} J_{ij}^{(x)}(r^*) u_i u_j \right),$$
where $J_{ij}^{(x)}(r)$ is $-2\frac{e^{r_i+r_j}}{m(x)}$ when $e_i\neq e_j\in E_x$, and $2\frac{e^{r_i}}{m(x)}\left(1-\frac{e^{r_i}}{m(x)}\right)$ when $j=i$.
For any $x\in V$
\[\begin{split}
    \sum_{e_i, e_j \in E_x} J_{ij}^{(x)}(r) u_i u_j
    &=\sum_{e_i \in E_x} J_{ii}^{(x)}(r)  u_i^2 +\sum_{e_i\neq e_j \in E_x} J_{ij}^{(x)}(r)  u_i u_j\\
    &=2\sum_{e_i \in E_x} \frac{e^{r_i}}{m(x)}\left(1-\frac{e^{r_i}}{m(x)}\right) u_i^2 -2\sum_{e_i\neq e_j \in E_x} \frac{e^{r_i+r_j}}{m(x)} u_i u_j\\
    &=2\sum_{e_i \in E_x} \frac{e^{r_i}}{m(x)}u_i^2 
    - 2\left(\sum_{e_i \in E_x}\frac{e^{2r_i}}{m(x)^2} u_i^2 +\sum_{e_i\neq e_j \in E_x}\frac{e^{r_i+r_j}}{m(x)} u_i u_j\right)\\
    &=2\sum_{e_i \in E_x} \frac{e^{r_i}}{m(x)}u_i^2 
    - 2\left(\sum_{e_i \in E_x}\frac{e^{r_i}}{m(x)} u_i \right)^2,
\end{split}\]
From $\sum_{e_i \in E_x}\frac{e^{r_i}}{m(x)}=1$, we have 
$$ \sum_{e_i, e_j \in E_x} J_{ij}^{(x)}(r) u_i u_j=\sum_{e_i,e_j \in E_x}\frac{e^{r_i+r_j}}{m(x)^2}(u_i-u_j)^2.  $$
Since $r^*\in\ell^\infty(E)$ and the degree is bounded by $D$, we obtain 
 \begin{equation*}\label{lowerbound_J}
     \frac{e^{r_i^*}}{\sum\limits_{e_k\in E_x} e^{r^*_{k}}}  \ge (De^{2\|r^*\|_\infty})^{-1} > 0.
 \end{equation*}
 This completes the proof.
\end{proof}
 Let  $u(t) = r(t) - r^*$ with $t>0$, where $r(t)$ is the solution to the Ricci flow \eqref{flow-equation_r} with the initial value $r(0)$. 
Then, by \eqref{eq:tylor}, the Ricci flow \eqref{flow-equation_r} can be transformed into
\begin{equation*}\label{eq:moflow}
    \frac{du_i(t)}{dt} =  - \sum_{e_k \in N_i} J_{ik}(r^*)  (u_k-u_i) - F_i(u)
\end{equation*}
with the initial value $u(0)$. 
Define the Lapalace operator on $C(E)$ as follows: for any  $f\in C(E)$,  
\[\Delta_J f_i:=\sum_{e_k\in N_i}(-J_{ik}(r^*))(f_k-f_i).\]
Then, the Ricci flow \eqref{flow-equation_r} can be rewritten as 
\begin{equation}\label{eq:moflow_v1}
    \frac{du_i}{dt} =  \Delta_J u_i - F_i(u).
\end{equation}
\begin{theorem}\label{th:3.6}
    Assume that $\|u(0)\|_{\ell^\infty}\leq 3\sqrt{3} \min\{1,(De^{2\|r^*\|_\infty})^{-2}\}$ and  the assumption $(A_1)$ holds. Then, the solution $u(t)$ of \eqref{eq:moflow_v1}  with the initial value $u(0)$ exists for $t \in [0, \infty)$ satisfying $\|u(t)\|_{\ell^2}$ is non-increasing and $$\int_0^\infty \mathcal{E}(u(t)) \, dt <\infty.$$
\end{theorem}
\begin{proof}
The proof is separated into two steps.\\
\textbf{Step 1. Suppose that the assumption $(A_1)$ holds. Then, for any given initial value $u(0) \in \ell^2(E)$, there exists $T_{\max}>0$ such that the equation \eqref{eq:moflow_v1} possesses a unique solution  $u \in C^1([0, T_{\max}), \ell^2(E))$.}

For any $u \in \ell^2(E)$,
$$ \|\Delta_J u\|_{\ell^2}^2 = \sum_{e_i \in E} \left| \sum_{e_k \in N_i} (-J_{ik}(r^*)) (u_k - u_i) \right|^2. $$
By  $|J_{ik}(r^*)|\leq 2$ and $(A_1)$, we have
$\sum_{e_k \in N_i} |J_{ik}(r^*)| \le 4(D-1)$. Notice that the facts $J_{ik}=J_{ki}$ for any $i\neq k$ and  $(u_k - u_i)^2 \le 2(u_k^2 + u_i^2)$, applying Cauchy-Schwarz inequality to get
\begin{align}\label{bound_laplacian}
    \|\Delta_J u\|_{\ell^2}^2 &\leq  8(D-1)\sum_{e_i \in E} \sum_{e_k \in N_i} |J_{ik}(r^*)| u_k^2 + 8(D-1) \sum_{e_i \in E}\sum_{e_k \in N_i}|J_{ik}(r^*)| u_i^2 \notag \\
    &=16(D-1) \sum_{e_i \in E}\sum_{e_k \in N_i}|J_{ik}(r^*)| u_i^2 \notag\\
    &\leq 64(D-1)^2\|u\|_{\ell^2}^2.
\end{align}
Actually, similiar to the Laplacian on $C(V)$, the finiteness of
$\sup_{x\in V}\sum_{y\sim x}J_{xy}$
 is equivalent to the boundedness of $\Delta_J$ on $\ell^2(E)$.

Denote $\Phi(u)=\Delta_J u - F(u)$. Let $X_R=\{u\in \ell^2(E): \|u\|_{\ell^2}\leq R\}$ with $R>0$. From \eqref{eq:F_bound} in Proposition \ref{lemma:3.52}, for any $u\in X_R$
\begin{equation}\label{eq:bound_phi}
    \|\Phi(u)\|_{\ell^2}\leq  \|\Delta_J u\|_{\ell^2}+ \|F(u)\|_{\ell^2}\leq C_1(D)\|u\|_{\ell^2}.
\end{equation}
Moreover, by \eqref{eq:F_bound_gradient} in Proposition \ref{lemma:3.52}, for any $u,v\in X_R$,
\begin{equation}\label{eq:gradient_phi}
\|\Phi(u)-\Phi(v)\|_{\ell^2}\leq \|\Delta_J u-\Delta_J v\|_{\ell^2}+\|F_i(u)-F_i(v)\|_{\ell^2}\leq C_2(D,R)\|u- v\|_{\ell^2}.
\end{equation}

Let $u(0)=u_0$, and denote $$\bar{B}_1(u_0) = \{ v \in \ell^2(E): \|v - u_0\|_{\ell^2}\le 1 \}.$$ 
Given $\tau=\min\{C_3^{-1}, (2C_4)^{-1}>0\}$, let $\|u\|_{\mathcal{Y}_\tau} := \sup_{t \in [0, \tau]} \|u(t)\|_{\ell^2}$, define $$\mathcal{Y}_\tau = C([0, \tau], \bar{B}_1(u_0)).$$ The space $(\mathcal{Y}_\tau, \|\cdot\|_{\mathcal{Y}_\tau})$ is complete.
For any $u \in \mathcal{Y}_\tau$, define 
$$\mathcal{T}(u)(t) := u_0 + \int_0^t \Phi(u(s)) ds.$$
From \eqref{eq:bound_phi} and \eqref{eq:gradient_phi}, we have for any $u \in \mathcal{Y}_\tau$, 
$$\|\Phi(u)\|_{\ell^2} \le \|\Phi(u_0)\|_{\ell^2} +\|\Phi(u)-\Phi(u_0)\|_{\ell^2} \leq C_1(D)\|u_0\|_{\ell^2}+ C_2(D,\|u_0\|_{\ell^2}+1).$$
Let $C_3(D,\|u_0\|_{\ell^2})=C_1(D)\|u_0\|_{\ell^2}+ C_2(D,\|u_0\|_{\ell^2}+1)$, $C_4(D,\|u_0\|_{\ell^2})=C_2(D,\|u_0\|_{\ell^2}+1)$.
It follows that for any $t\in(0,\tau)$,
$$\|\mathcal{T}(u)(t) - u_0\|_{\ell^2} = \left\| \int_0^t \Phi(u(s)) ds \right\|_{\ell^2}\le \int_0^t \|\Phi(u(s))\|_{\ell^2} ds \le C_3 t \le C_3 \tau.$$
The condition $\tau \le C_3^{-1}$ ensures that 
$\mathcal{T}(u)(t) \in \bar{B}_1(u_0)$, which implies that $\mathcal{T}:\mathcal{Y}_\tau \rightarrow \mathcal{Y}_\tau$.
Moreover, for any $u,v\in  \mathcal{Y}_\tau$, due to \eqref{eq:gradient_phi},
$$\|\mathcal{T}(u)(t) - \mathcal{T}(v)(t)\|_{\ell^2} \le \int_0^t \|\Phi(u(s)) - \Phi(v(s))\|_{\ell^2} ds\leq C_2(D,\|u_0\|_{\ell^2}+1)\tau. $$
If choosing $\tau \le (2C_4)^{-1}$, $\mathcal{T}$ is a contraction map. Therefore, by the contraction mapping theorem, there is a solution $u \in C^1([0, \tau]; \ell^2(E))$ to \eqref{eq:moflow_v1}  for any initial value $u(0)\in \ell^2(E)$. Therefore, there exists a $T_{\max}>0$ such that the solution can be extended to $[0, T_{\max})$.\\
\textbf{Step 2.  Assume that $\|u(0)\|_{\ell^\infty}\leq 3\sqrt{3} \min\{1,(De^{2\|r^*\|_\infty})^{-2}\}$. The solution $u(t)$ to  \eqref{eq:moflow_v1} can be extended to $L^\infty([0, \infty), \ell^2(E))$. And,  $\|u(t)\|_{\ell^2}$ is non-increasing in $t\in[0,\infty)$ and  $$\int_0^\infty \mathcal{E}(u(t)) \, dt \le \|u(0)\|_{\ell^2}^2.$$}

Frist, we prove $\|u(t)\|_{\ell^\infty}\leq M$ on $[0,T_{\max})$ by providing  $\|u(0)\|_{\ell^\infty}\leq M$ with $M<3\sqrt{3} (De^{2\|r^*\|_\infty})^{-2}$.  Define $M(t) = \sup_{e_i \in E} u_i(t)$ and $m(t) = \inf_{e_i \in E} u_i(t)$. Since $u(t) \in \ell^2(E)$, it follows that if $M(t) > 0$, the supremum is necessarily attained at some finite index $i(t)$. 
Suppose that at some $t_0$,  $M(t_0) = M$ and is attained at $i_0$, i.e., $u_{i_0}(t_0) = M$. 
The evolution equation \eqref{eq:moflow_v1} at $t_0$ and $i_0$ is 
$$\frac{du_{i_0}}{dt}(t_0) = \Delta_J u_{i_0}- F_{i_0}(u)$$
Utilizing  \eqref{eq:F_upper} in Lemma \ref{lemma:3.5},
$$\frac{du_{i^*}}{dt}(t_0) \le \sum_{k \in N(i^*)} (u_k(t_0)  - u_{i_0}(t_0) )\left[ -J_{i_0k} (r^*) - \frac{\sqrt{3}}{18} (u_{i_0}(t_0)  - u_k(t_0) ) \right].$$
Suppose that $ \|u(t_0)\|_{\ell^\infty} = M$, we have $u_{i_0} - u_k \le 2M$.
Combining with \eqref{lowerbound_J},  the condition $M<3\sqrt{3} (De^{2\|r^*\|_\infty})^{-2}$ ensures $$ -J_{i_0k} (r^*) - \frac{\sqrt{3}}{18} (u_{i_0}(t_0)  - u_k(t_0) ) > (De^{2\|r^*\|_\infty})^{-2}  - \frac{\sqrt{3}}{18} \cdot 2M > 0,$$
yielding $\frac{du_{i_0}}{dt}(t_0) \le 0$. A similar argument applies to the lower bound of the minimum showing $\frac{du_{i^*}}{dt}(t_0)  \ge 0$ whenever $m(t_0) = -M < 0$. Consequently, the $\ell^\infty$-norm cannot strictly increase once it reaches $M$. 
This leads to $$\|u(t)\|_{\ell^\infty} \le \|u(0)\|_{\ell^\infty} \le M < 3\sqrt{3} (De^{2\|r^*\|_\infty})^{-2}, \quad \forall t \in [0, T_{\max}).$$
Next, we prove $T_{\max}=\infty$. Let $u(t)\in \ell^2(E)$ be the solution to  \eqref{eq:moflow_v1} on $[0,T_{\max})$. 
By the evolution equation \eqref{eq:moflow_v1}, 
$$\frac{1}{2} \frac{d}{dt} \|u(t)\|_{\ell^2}^2 =  \sum_{e_i \in E} u_i \frac{du_i(t)}{dt}=- \mathcal{E}(u(t)) - \sum_i u_i(t) F_i(u(t)).$$
From \eqref{eq:F_upper} in Lemma \ref{lemma:3.5},  we have
\[\left| \sum_{e_i \in E} u_i F_i(u) \right| 
\le \frac{\sqrt{3}}{18}\sum_{e_i \in E} |u_i| \left( \sum_{e_k \in N_i} (u_k - u_i)^2 \right)
\le \frac{\sqrt{3}}{18}\|u\|_{\ell^\infty} \mathcal{E}(u),\]
which yields
\begin{equation}\label{eq:ell2}
    \frac{1}{2} \frac{d}{dt} \|u(t)\|_{\ell^\infty}\le \left(\frac{\sqrt{3}}{18}\|u(t)\|_{\ell^\infty }-1\right)\mathcal{E}(u(t)) \le -\frac{1}{2}\mathcal{E}(u(t))(\leq0)
\end{equation} 
due to $\|u(t)\|_{\ell^\infty}\leq 3\sqrt{3}$ by  providing  $\|u(0)\|_{\ell^\infty}\leq 3\sqrt{3} \min\{1,(De^{2\|r^*\|_\infty})^{-2}\}$.
This implies that $\|u(t)\|_{\ell^2}^2$ is non-increasing in $[0,T_{\max})$.
And, $\|u(t)\|_{\ell^2} \le \|u(0)\|_{\ell^2}$ for all $t \in [0, T_{\max})$, which implies that the solution never blows up at a finite time, namely $T_{\max} = \infty$.

Moreover, by virtue of \eqref{eq:moflow_v1}, the similar proof ensures that \eqref{eq:ell2}
 holds  for all $t \ge 0$. It follows that $\|u(t)\|_{\ell^2}^2$ is non-increasing in $t\in [0,\infty)$. Moreover, integrating \eqref{eq:ell2} over  $[0, t]$ and using the fact of $\|u(t)\|_{\ell^2}^2 \ge 0$, it yields
 $$\int_0^t \mathcal{E}(u(\tau)) \, d\tau \le \|u(0)\|_{\ell^2}^2 - \|u(t)\|_{\ell^2}^2 \le \|u(0)\|_{\ell^2}^2.$$
 Taking the limit as $t \to \infty$ to get 
 $$\int_0^\infty \mathcal{E}(u(t)) \, dt \le \|u(0)\|_{\ell^2}^2.$$
 This ends the proof.
\end{proof}
Before proving Theorem \ref{main5}, we need the following Barbalat's Lemma and another Lemma from Ge, Hua and Zhou \cite{ge2025combinatorial}.
\begin{lemma}[Barbalat's Lemma]\label{lemma:Barbalat}
    Let $f: [0, \infty) \to \mathbb{R}$ be a uniformly continuous function. If the limit $\lim_{t \to \infty} \int_0^t f(s) \, ds$ exists and is finite, then $\lim_{t \to \infty} f(t) = 0$.
\end{lemma}
\begin{lemma}\label{lemma:3.7}
    Let $\{u_n\}_{n=1}^\infty$ be a sequence of $\ell^2$ functions satisfying that $\|u_n\|_{\ell^2}$ is uniformly bounded and 
    $$ \mathcal{E}(u_n) \to 0, \quad n \to \infty.$$
    Then
    $$\|u_n\|_{\ell^\infty} \to 0, \quad n \to \infty.$$
\end{lemma}
\begin{proof}[Proof of Theorem \ref{main5}] 
Notice that 
\(-\sum_{j \in E} J_{ij}(r^*) u_j =\Delta_Ju_i.\)
Then
\[\mathcal{E}(u)=-\sum_{e_i\in E}u_i\Delta_J u_i.\]
Considering the following series, by $J_{ij} = J_{ji}$, we have

$$\phi_i(t):=\sum_{j \in E} J_{ij}(r^*) \left( \frac{du_i}{dt} u_j + u_i \frac{du_j}{dt} \right)= 2 \frac{du_i}{dt} \left( \sum_{j \in E} J_{ij}(r^*) u_j \right).$$
Substituting the evolution equation \eqref{eq:moflow_v1} into the above equality, yields
$$\phi_i(t)=-2 \big(\Delta_J u_i - F_i(u) \big) \big( \Delta_J u_i \big)= -2 (\Delta_J u_i)^2 + 2  F_i(u) (\Delta_J u_i).$$
From the boundeness of $\Delta_J$ (see \eqref{bound_laplacian}) , \eqref{eq:F_bound} in Lemma \ref{lemma:3.52} and  the fact that $
    \|u(t)\|_{\ell^2} \le  \|u(0)\|_{\ell^2}
$ in Theorem \ref{th:3.6},  we have
$$|\phi_i(t)|\le 2\|\Delta_J u\|_{\ell^2} \left( \|\Delta_J u\|_{\ell^2} + \|F(u)\|_{\ell^2} \right)\leq C_1\|u(0)\|_{\ell^2}^2+C_2\|u(0)\|_{\ell^2}^3=:C_0,$$
where $C_1$ and $C_2$ depend on $D.$
Then, the series $\sum_{e_i\in E}\phi(t)$ is uniform convergence, which means
\[\frac{d}{dt} \mathcal{E}(u(t))=\frac{1}{2}\sum_{e_i\in E}\phi_i(t).\]
Moreover, 
$$\left| \frac{d}{dt} \mathcal{E}(u(t)) \right|  \le \|\Delta_J u\|_{\ell^2} \left( \|\Delta_J u\|_{\ell^2} + \|F(u)\|_{\ell^2} \right)\leq \frac{C_0}{2}.$$
This ensures the uniform continuity of $\mathcal{E}(u(t))$ on $[0, \infty)$. According to Barbalat’s Lemma (see Lemma \ref{lemma:Barbalat}), combined with  $\int_0^\infty \mathcal{E}(u(t)) \, dt < \infty$ in Theorem \ref{th:3.6}, we  obtain
$$\lim_{t \to \infty} \mathcal{E}(u(t)) = 0.$$
Since $\|u(t)\|_{\ell^2}\leq \|u(0)\|_{\ell^2}$, by Lemma \ref{lemma:3.7}, we have $\|u(t)\|_{\ell^\infty}\rightarrow0$ as $t\rightarrow\infty$. Moreover,
by Theorem \ref{main3}, the solution $r(t)$ to the flow \eqref{flow-equation3} is unique. Then 
\[r(t)\rightarrow r^*, \quad t\rightarrow \infty.\]
 This completes the proof.
\end{proof}

\subsection{Proof of Theorem \ref{main4}}
Similar to  Proposition 3.2 in \cite{2002Combinatorial} and Lemma 3.2 in \cite{ge2025combinatorial}, we need the following maximum principle.
\begin{lemma}[Maximum principle(I)] \label{max_min}
For any $l\in\mathbb{N}^+$ and $t\geq0$, let $M_0^l(t)=\max \{\max_{e_i\in E^l} (\kappa_i(r^l(t))-\kappa_i^*),0\}$, $m_0^l(t)=\min \{\min_{e_i\in E^l} (\kappa_i(r^l(t))-\kappa_i^*),0\}$. Then, under the Ricci flow \eqref{flow_exhaustion}, $M_0^l(t)$ is non-increasing  and $m_0^l(t)$ is non-decreasing in $t$. 
\end{lemma}
\begin{proof}
    Let $e_i=(x,y)$, under the Ricci flow \eqref{flow_exhaustion}, due to \eqref{eq:par_kappa}, the evolution equation for the curvature satisfies 
\begin{align*}
    \frac{d}{dt}\kappa_i(r^l(t))
    &= \sum_{e_j\in E}\frac{\partial \kappa_i}{\partial r_j^l} \frac{dr^l_j}{dt}\\
    &=-J_{ii}(r^l)(\kappa_i(r^l(t))-\kappa^*_i)
    -\sum_{j\sim i, e_j\in E^l}J_{ij}(r^l)(\kappa_j(r^l(t))-\kappa^*_j)\\
    &=-\sum_{j\sim i, e_j\in E^l}J_{ij}(r^l)[(\kappa_j(r^l(t))-\kappa^*_j)-(\kappa_i(r^l(t))-\kappa^*_i)]+\sum_{j\sim i, j\in \partial E^l}J_{ij}(r^l)\cdot (\kappa_i(r^l(t))-\kappa^*_i).
\end{align*}
Notice that $-J_{ij}=-J_{ji}(>0)$ for any $j\sim i$ and $\sum_{j\sim i, j\in \partial E^l}J_{ij}\leq 0$. It is natural to get the conclusions.
\end{proof}
In addition, another maximum principle is needed. 
\begin{lemma}[Maximum principle(II)] \label{max}
For any $l\in\mathbb{N}^+$ and $t\geq0$, let $u^l_i(t)=r_i^l(t)-r_i^*$ for any $e_i\in E^l\cup \partial E^l$. Denote  $M^l(t)= \max_{e_i \in E^l \cup \partial E^l} u_i^l(t)$ and $m^l(t)= \min_{e_i\in E^l \cup \partial E^l} u_i^l(t)$. Then, under the Ricci flow \eqref{flow_exhaustion}, $M^l(t)$ is non-increasing  and $m^l(t)$ is non-decreasing in $t$. 
\end{lemma}
\begin{proof}
Let 
$$I_t= \{ e_j\in E^l \cup \partial E^l : u_j^l(t) = M^l(t) \}.$$
It is sufficiently to prove that $\frac{d}{dt}u_i^l(t) \le 0$ for any $i \in I_t$. For any $t>0$, let  $k\in I_t$ be the maximum on $E^l \cup \partial E^l$, then $u_k^l(t) \ge u_j^l(t)$ for any  $e_j\in E^l \cup \partial E^l $, namely,
 $$r_j^l - r_k^l \le r_j^* - r_k^*.$$ 
When $e_k \in E^l$,  any neighbor of $e_k$ is in $E^l \cup \partial E^l$. According to the definition of $\kappa$ on edge $e_k=(x,y)$ on graphs with girth at least 6, as follows
 \begin{equation*}\label{kappa_r}
\kappa_{xy}=2\left(\frac{1}{\sum_{e_j\in E_x}e^{r_j-r_k}}+\frac{1}{\sum_{e_j\in E_y}e^{r_j-r_k}}\right)-2,
\end{equation*}
we have $\kappa_k(r^l(t)) \ge \kappa_k^*$. Substituting this into the Ricci flow  \eqref{flow_exhaustion}, it yields
$$\frac{d}{dt}u_k^l(t) = \kappa_k^* - \kappa_k(r^l(t)) \le 0.$$
When $e_k \in \partial E^l$, by the Dirichlet boundary condition $r_k^l(t)=r_k^l(0)$ for any $t>0$, 
$$\frac{d}{dt}u_k^l(t) = 0.$$
A similar argument applies to the minimum function $m^l(t)$.
This completes the proof.
\end{proof}
Let $r(t)$ be the solution of the Ricci flow \eqref{flow-equation3} with the initial value $r(0)$ obtained by the convergence of a subsequence of $\{r^{l}(t)\}_{l=1}^\infty$, which are the solutions to the Ricci flow with boundary \eqref{flow_exhaustion}.
\begin{corollary}\label{bounds}
    Suppose that $r^*\in \ell^\infty(E)$. Then, for any $t>0$ and each $e_i\in E$, $r_i^l(t)\leq \sup_{e_j\in E}r_j(0)+2\|r^*\|_\infty$ and $r_i^l(t)\geq \inf_{e_j\in E}r_j(0)-2\|r^*\|_\infty$. Moreover,  $r(t)$ also has the same uniform upper (or lower) bound.
\end{corollary}
\begin{proof}
    From  Lemma \ref{max}, for any $e_i\in E^l\cup \partial E^l,$
    \[u_i^l(t)\leq M^l(t)\leq  M^l(0)=\max_{e_j \in E^l \cup \partial E^l} (r_j^l(0)-r^*_e)\leq\sup_{e_j\in E}r_j(0)+\|r^*\|_\infty.\]
    Therefore,
    \[r_i^l(t)= u_i^l(t)+r_i^*\leq \sup_{e_j\in E}r_j(0)+2\|r^*\|_\infty.\]
A similar argument applies to the lower bound of $r_i^l(t)$.    This completes the proof.
\end{proof}
We are ready to prove Theorem \ref{main4}.
\begin{proof}[Proof of Theorem \ref{main4}]
    Let $r(0)$ be the initial value such that  $\kappa_i(r(0)) \le \kappa^*_i$ 
    for any $e_i\in E$.  By Lemma \ref{max_min}, for all $i \in E^l$ and $l\in \mathbb{N}^+$, 
 $$\kappa_i(r^{l}(t)) \le \kappa^*_i,\quad \forall t>0.$$
 Therefore, $\kappa_i(r(t)) \le \kappa^*_i$ for all $(i, t) \in E \times [0, \infty)$.
Combining with the Ricci flow \eqref{flow-equation_r},  it implies that $r_i(t)$ is non-decreasing for each $e_i \in E$. Assume that $r^*\in \ell^\infty(E)$ and $\sup_{e_j\in E}r_j(0)<+\infty$. Due to Corollary \ref{bounds}, $r_i(t)$ 
has a limit $r_i(\infty) \in (-\infty, \sup_{e_j\in E}r_j(0)+2\|r^*\|_\infty]$. Similarly, if $\kappa_i(r(0)) \ge \kappa^*_i$,
  $r_i(t)$  has a limit $r_i(\infty) \in [\inf_{e_j\in E}r_j(0)-2\|r^*\|_\infty,+\infty)$. 

 If $r(t)$ converges, $\kappa_i(r(t))$ must converge to some $\kappa_i(\infty)$. Without loss of generality, we consider the case that $\sup_{e_j\in E}r_j(0)<+\infty$ and $\kappa(r(0))\leq \kappa^*$. Now we prove that $\kappa_i(\infty) = \kappa_i^*$ for any $e_i\in E$.  Suppose it is not true, i.e. $\kappa_i(\infty)<\kappa_i^*$ for some $e_i \in E$. Then there exists a $T > 0$ such that  for all $t > T$, $$\kappa_j(t)-\kappa_j^*\le \frac{\kappa_j(\infty)-\kappa_j^*}{2} < 0.$$ Therefore, from the Ricci flow \eqref{flow-equation_r}, $$\frac{dr_j(t)}{dt} \ge -\frac{\kappa_j(\infty)-\kappa_j^*}{2} (> 0)$$ for $t > T$, which implies that $r_j(t)$ will exceed $\sup_{e_j\in E}r_j(0)+2\|r^*\|_\infty$ within finite time. It leads to a contradiction. Therefore, $\kappa(\infty) = \kappa^*$ on $E$. 
\end{proof}

{\bf Acknowledgments:} 
  We are grateful to Florentin M\"unch for his insightful discussions regarding Ricci flows on infinite graphs.  B. Hua is supported by NSFC, no.12371056. Y. Lin is supported by NSFC, no.12471088. S. Liu is supported by NSFC, no.12001536, 12371102.

\bibliographystyle{plain}
\bibliography{citations}
\end{document}